\documentclass{birkjour}
\usepackage{amsmath,amsfonts,latexsym,amssymb,url}
\usepackage{graphicx}
\usepackage{hyperref,pdfsync}
\usepackage{amsmath,txfonts,pifont,bbding,pxfonts,manfnt}
\usepackage[active]{srcltx}
\usepackage{wasysym,pstricks}
\usepackage{enumerate}

\newcommand{\R}{{\mathbb R}} 

\newtheorem{theorem}{Theorem}

\newtheorem{lemma}[theorem]{Lemma}
\newtheorem{example}[theorem]{Example}
\newtheorem{proposition}[theorem]{Proposition}

\newtheorem{remark}[theorem]{Remark}

\newtheorem{problem}[theorem]{Problem}

\begin{document}

%
%
%
%
%
%
%
%
%

\title[Functional Equations and the Cauchy Mean Value Theorem]{Functional Equations and the Cauchy Mean Value Theorem}

\author{Zolt\'an M. Balogh}

\address{%
Institute of Mathematics\\
University of Bern\\
Sidlerstrasse 5\\
CH 3012\\
Switzerland}

\email{zoltan.balogh@math.unibe.ch}

\thanks{The authors were supported by the Swiss National Science Foundation,
SNF, grant no. 200020$\_$146477.}

\author{Orif O. Ibrogimov}

\address{%
Institute of Mathematics\\
University of Bern\\
Sidlerstrasse 5\\
CH 3012\\
Switzerland}
\email{orif.ibrogimov@math.unibe.ch}

\author{Boris S. Mityagin}

\address{%
Department of Mathematics\\
The Ohio State University\\
231 W 18th Ave. MW 534\\
Columbus Ohio 43210\\
 U.S.A.}
\email{mityagin.1@osu.edu}



\date{\today}
\dedicatory{Dedicated to Professor J\"{u}rg R\"{a}tz
}

\begin{abstract}
The aim of this note is to characterize all pairs of sufficiently smooth functions for which the mean value in the Cauchy Mean Value Theorem is taken at a point which has a well-determined position in the interval. As an application of this  result, a partial answer is given to a question posed by Sahoo and Riedel.
\end{abstract}

\maketitle

%
\section{Introduction}
Given two differentiable functions $F,G:\R \to \R$, the Cauchy Mean Value Theorem (MVT) states that for any interval $[a,b]\subset \R$, where $a<b$, there exists a point $c$ in $(a,b)$ such that 
\begin{equation} \label{Cauchy}
[F(b)-F(a)]\,g(c) = [G(b)-G(a)]\,f(c),
\end{equation}
where $f= F'$ and $g=G'$. A particular situation is the Lagrange MVT when $G(x)=x$ is the identity function, 
in which case \eqref{Cauchy} reads as
\begin{equation} \label{Lagrange}
F(b)-F(a) = f(c)(b-a).
\end{equation}

The problem to be investigated in this note can be formulated as follows.
\begin{problem}
\label{the-problem}
Find all pairs $(F,G)$ of differentiable functions $F, G :\R\to\R$ satisfying the following equation 
\begin{equation}
\label{eqn:cauchy.gen}
[F(b)-F(a)] \, g(\alpha a+\beta b) = [G(b)-G(a)] \, f(\alpha a+\beta b)  
\end{equation}
for all $a,b \in \R$, $a<b$, where $f=F'$, $g=G'$, $\alpha, \beta\in (0,1)$ are fixed and $\alpha+\beta=1$.
\end{problem}

For the case of the Lagrange MVT with $c=\frac{a+b}{2}$, this problem was considered first by Haruki \cite{Haruki} and independently by Acz\'el \cite{Aczel}, proving that the quadratic functions are the only solutions to \eqref{Lagrange}. This problem can serve as a starting point for various functional equations \cite{Sahoo-Riedel}. More general functional equations have been considered 
even in the abstract setting of groups by several authors including Kannappan \cite{Kannappan}, Ebanks \cite{Ebanks}, Fechner-Gselmann \cite{Fechner-Gselmann}. On the other hand, the result of Acz\'el and Haruki has been generalized for higher order Taylor expansion by Sablik \cite{Sablik}. 

For the more general case of the Cauchy MVT much less is known. We mention Aumann \cite{Aumann} illustrating the geometrical significance of this equation and the recent contribution of P\'ales \cite{Pales} providing the solution of a related equation under additional assumptions. In this note we provide a different approach to the Cauchy MVT.
 As it will turn out, the most challenging situation corresponds to $c=\frac{a+b}{2}$ in which case our main result is the following:

\begin{theorem} 
\label{thm:sym.cauchy.main}
Assume that $F,G: \R \to \R$ are three times differentiable functions with derivatives $F'=f$, $G'=g$ such that
\begin{equation}
\label{eqn:cauchy.sym1}
[F(b)-F(a)]\,g\bigl(\frac{a+b}{2}\bigr) = [G(b)-G(a)]\,f\bigl(\frac{a+b}{2}\bigr) 
\end{equation}
for all $a,b\in\R$. Then one of the following possibilities holds: 

\begin{enumerate}[(a)]

\item $\{1, F, G\}$ are linearly dependent on $\R$;

\item $F,G \in \mathrm{span} \{1,x,x^2\}$, $x\in\R$;

\item there exists a non-zero real number $\mu$ such that 
\[
F,G \in \mathrm{span} \{1,e^{\mu x}, e^{-\mu x}\}, \quad x\in\R;
\]

\item there exists a non-zero real number $\mu$ such that 
\[
F,G \in \mathrm{span} \{1,\sin(\mu x), \cos(\mu x)\}, \quad x\in\R.
\]
\end{enumerate} 
\end{theorem}

The paper is organized as follows. In Section 2 we consider the problem first for the known case of the Lagrange mean value theorem  as an illustration of our method. In Section 3 we provide a preliminary result that will allow to pass local information to a global one about the pairs of differentiable functions $(F,G)$ 
satisfying \eqref{eqn:cauchy.gen}. In Sections 4, 5 we consider the asymmetric ($\alpha\neq\beta$) and symmetric ($\alpha=\beta=1/2$) cases, respectively. Section 6 is for final remarks. Here we also provide a partial result to an open problem by Sahoo and Riedel which corresponds to a more general version of \eqref{eqn:cauchy.gen}.

\section{The Lagrange MVT with fixed mean value} 

Note that every $c \in (a,b)$ can be written uniquely as $c=\alpha a+\beta b$ for some $\alpha,\beta\in (0,1)$ with $\alpha+\beta=1$. It is easy to check that \eqref{Lagrange} holds for all $a,b\in\R$ with fixed $\alpha \neq 1/2$ if $F$ is a linear function, and with  $\alpha = 1/2$ if $F$ is a quadratic function. We claim that the converse of this statement is also true. As mentioned earlier, there are various proofs of the latter in the literature, see for example \cite{Aczel}, \cite{Haruki} , \cite{Sahoo-Riedel}. Nevertheless, we give here  a short and self-contained argument mainly to illustrate our approach to the more general case of the Cauchy MVT.

\begin{proposition}\label{lagrange}
Let $\alpha \in (0,1)$ be fixed and $\beta= 1-\alpha$. 
Assume that $F:\R \to \R$ is a continuously differentiable function with $F\rq{}=f$  such that 
\begin{equation} \label{convex-comb}
F(b)-F(a)= f(\alpha b+ \beta a) (b-a) \quad \text{for all} \quad a,b\in \R \quad \text{with} \quad a<b.
\end{equation} 
Then the following statements hold:
\begin{enumerate} 
\item if $\alpha \neq 1/2$ then $F$ is a linear function;
\item if $\alpha = 1/2$ then $F$ is a quadratic function.
\end{enumerate}
\end{proposition}

\begin{proof}
Let us denote by $\alpha b+ \beta a= x$ and $b-a= h$. Then \eqref{convex-comb} reads as 
\begin{equation} \label{new-lagrange}
F(x+\beta h) -F(x-\alpha h)= f(x)\,h \quad \text{for all} \quad x \in \R, h>0.
\end{equation} 
From this equation it is apparent that $f=F'$ is differentiable as a linear combination of two differentiable functions and thus $F$ is twice differentiable. By induction, it follows that $F$ is 
infinitely differentiable.

Differentiating \eqref{new-lagrange} with respect to $h$, we obtain the relation
\begin{equation} \label{diff-lagrange}
\beta f(x+\beta h) + \alpha f(x-\alpha h) = f(x), \quad 
x \in \R, h>0.
 \end{equation}
Again, we differentiate \eqref{diff-lagrange} with respect to $h$ and find that
\begin{equation*}
\beta^2 f'(x+\beta h) -\alpha^2 f'(x-\alpha h) =0, \quad 
x \in \R, h>0.
\end{equation*}
Since $f'$ is continuous, letting $h \searrow 0$, we obtain 
$$ (\beta^2-\alpha^2) f'(x)=(1-2\alpha)f'(x)=0 \quad \text{for all} \quad x\in \R.$$
If $\alpha \neq 1/2$, this implies that $f'=0$ identically. Therefore $f$ is constant and thus $F$ is a linear function, proving the first statement. 

If $\alpha = 1/2$, then the equation \eqref{diff-lagrange} reads as
\begin{equation*}
f\bigl(x+\frac{h}{2}\bigr) +f\bigl(x-\frac{h}{2}\bigr)= 2 f(x), \quad 
x \in \R, h>0,
\end{equation*}
and twice differentiation with respect to $h$ leads to
\begin{equation*}
f''\bigl(x+\frac{h}{2}\bigr) + f''\bigl(x-\frac{h}{2}\bigr) = 0, \quad 
x \in \R, h>0.
\end{equation*}
Now letting $h \searrow 0$, we get $f''(x)=0$ for all $x\in\R$, so $f$ is linear and $F$ is a quadratic function, proving the second statement.
\end{proof} 


\section{The Cauchy MVT with fixed mean value}

Let us introduce the sets
\begin{align}\label{U_f,U_g}
U_f := \{ x \in \R: f(x) \ne 0 \}, \quad U_g := \{ x \in \R: g(x) \ne 0 \},
\end{align}
and also their complements $Z_f := \R \setminus U_g$ and $Z_g := \R \setminus U_f$. Observe that if $U_g$ is empty, i.e. $G$ is constant on $\R$, then \eqref{eqn:cauchy.gen} holds for trivial reasons (both sides are identically zero) for arbitrary differentiable function $F$. Of course, we can change the roles of $G$ and $F$ and claim: if $F$ is constant then \eqref{eqn:cauchy.gen} holds for any differentiable function $G$. Assume therefore that $U_g\neq\emptyset$. Then there is a sequence of mutually disjoint open intervals $\{I_\sigma\}_{\sigma\in\Sigma}$, $\Sigma\subset\mathbb{N}$, such that
\begin{align}
\label{rep:U_g=sum.of.itvs}
\displaystyle U_g=\bigcup_{\sigma\in\Sigma}I_{\sigma}.
\end{align}
%
\begin{proposition}
\label{prop:cauchy.gen.U_g<>empty}
If $U_g\neq\emptyset$ but $U_f\cap U_g = \emptyset$, then $U_f=\emptyset$, i.e. $f\equiv 0$ on $\R$ and thus $F$ is constant.
\end{proposition}
\begin{proof}
By assumption, there is a non-empty interval $(p,q)\subset U_g$ such that $g(x)\neq 0$ on $(p,q)$, but $f(x)=0$ for all $x\in[p,q]$. Then with the change of variables $h=b-a$, $x=\alpha a+\beta b$, \eqref{eqn:cauchy.gen} yields
\begin{equation} \label{constancy} 
F(x+\alpha h) - F(x-\beta h) = 0 \quad \text{for all} \quad  x\in (p,q), \, h>0.
\end{equation} 

Denoting by $y=x+\alpha h$ for $x\in[p,q]$ and $h>0$, we get $F(y)-F(y-h) = 0$ if $(h,y)$ lies within the semi--strip $L := \bigl\{(h,y):\, h>0, \, p+\alpha h<y<q+\alpha h\bigr\}$. 
\begin{figure}[h!]
  \caption{}
  \centering
    \includegraphics[width=0.5\textwidth]{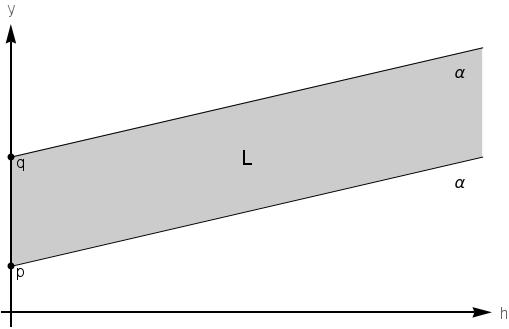}
\end{figure}

\smallskip
\noindent
Then, for $y>p$ choosing $h>0$ such that $(h,y)\in L$, we have
\begin{align*}
\frac{\partial}{\partial y}F(y) = \frac{\partial}{\partial y}F(y-h) &= -\frac{\partial}{\partial h}F(y-h)\\
                                                                   &= -\frac{\partial}{\partial h}F(y) =0, 
\end{align*}
so $F(y)$ is a constant, say $F(y)=F\bigl(\frac{p+q}{2}\bigr)$ for $y>p$. However, by \eqref{constancy}, we have $F(q+\alpha h) = F(q-\beta h)$ and thus $F(y)$ is the same constant for all $y<q$. Therefore, $f(y)=F'(y)=0$ for all $y\in\R$.
\end{proof}
%

Proposition \ref{prop:cauchy.gen.U_g<>empty} shows that the condition $U_f\cap U_g=\emptyset$ holds only if at least one of the sets $U_f$ and $U_g$ is empty. Then we have the simple cases described as in the beginning of the section.

%
\begin{proposition} 
\label{prop:X}
%
Let $(F,G)$ be a solution of the Problem~\ref{the-problem} satisfying 
\begin{align}
\label{Uf_disj_Ug}
U_f\cap U_g \neq \emptyset,
\end{align}
and consider the representation \eqref{rep:U_g=sum.of.itvs}. If $\{F, G, 1\}$ are linearly dependent as functions on $I_{\sigma}$ for every $\sigma\in\Sigma$, then $\{F, G, 1\}$ are linearly dependent on $\R$.
\end{proposition}
\begin{proof}
For $\sigma_1, \sigma_2 \in \Sigma$ with $\sigma_1 \neq \sigma_2$, consider the intervals $I_{\sigma_1} := (p_1,q_1)$, $I_{\sigma_2}:=(p_2,q_2)$ with 
\begin{align}\label{pqpq}
p_1<q_1\leq p_2<q_2,
\end{align}
and assume that $\{F, G, 1\}$ are linearly dependent on $I_{\sigma_1}$ and $I_{\sigma_2}$. Then it follows that there are constants $A_1$, $A_2$, $B_1$, $B_2 \in\R$ such that 
\begin{align}\label{FA_1B_1}
F(x) &= A_1G(x)+B_1, \, x\in I_{\sigma_1},\\ \label{FA_2B_2}
     &= A_2G(x)+B_2, \, x\in I_{\sigma_2}. 
\end{align}
With the change of variables $h=b-a$, $x=\alpha a+\beta b$, \eqref{eqn:cauchy.gen} yields
\begin{equation*}
[F(x+\alpha h)-F(x-\beta h)]\, g(x) = [G(x+\alpha h)-G(x-\beta h)]\, f(x)
\end{equation*}
for all $x\in\R$ and $h>0$. Since $f(x)=A_2g(x)$ if $x\in I_{\sigma_2}$ by \eqref{FA_2B_2} and $g(x)\neq 0$ for $x\in I_{\sigma_2}$, we have
\begin{align}\label{FFA_2}
F(x+\alpha h)-F(x-\beta h) = A_2[G(x+\alpha h)-G(x-\beta h)], 
\quad x\in I_{\sigma_2}, h>0.
\end{align}
If at the same time $x-\beta h\in I_{\sigma_1}$, then $F(x-\beta h)=A_1G(x-\beta h)+B_1$ by \eqref{FA_1B_1}. Inserting this value into \eqref{FFA_2}, we obtain
\begin{align}\label{FA_2B_1}
F(x+\alpha h) = A_2G(x+\alpha h)+(A_1-A_2)G(x-\beta h)+B_1
\end{align}
for
\begin{align}\label{I^2I^1h}
x\in I_{\sigma_2}, \, x-\beta h\in I_{\sigma_1}, \, h>0.
\end{align}
Put $y=x+\alpha h$, then $x-\beta h=y-h$, and \eqref{I^2I^1h} means that $(h,y)$ lies within the parallelogram $
\Pi := \bigl\{(h,y): \, p_2+\alpha h<y<q_2+\alpha h, \, p_1+h<y<q_1+h\bigr\}$.

\begin{figure}[h!]
  \caption{}
  \centering
    \includegraphics[width=0.5\textwidth]{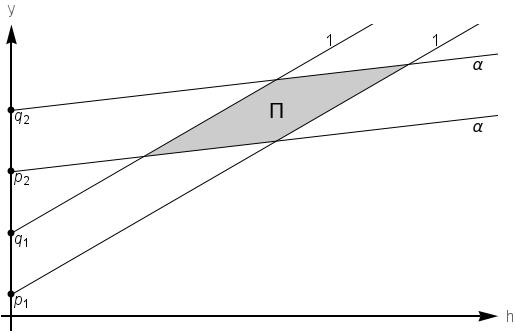}
\end{figure}

\smallskip
\noindent
Since $\beta\in(0,1)$, \eqref{pqpq} guarantees that $\Pi\neq\emptyset$, 
and \eqref{FA_2B_1} implies 
\begin{equation*}
F(y)= A_2G(y)+(A_1-A_2)G(y-h)+B_1 \quad \text{for all} \quad (h,y)\in\Pi.
\end{equation*}
Therefore, at any point of $\Pi$, we have
\begin{align*}
0=\frac{\partial}{\partial h}F(y) = -(A_1-A_2)G'(y-h) = (A_2-A_1)g(y-h).
\end{align*}
But $y-h\in I_{\sigma_1}$ by \eqref{I^2I^1h}, so $g(y-h)\neq 0$ and thus
\begin{align}\label{A_1=A_2}
A_2-A_1=0.
\end{align}
So far our analysis says nothing about $B_1$, $B_2$ in \eqref{FA_1B_1}, \eqref{FA_2B_2} but since $\sigma_1, \sigma_2\in\Sigma$ were arbitrary, \eqref{A_1=A_2} together with \eqref{FA_1B_1}, \eqref{FA_2B_2} imply 
\begin{align}\label{fAg}
f(x)=Ag(x) \quad \text{for some constant} \quad A\in\R \quad \text{and all} \quad x\in U_g.
\end{align}
On the other hand, by changing the roles of $F$ and $G$ in the above analysis, we come to the conclusion that 
\begin{align}\label{gKf}
g(x)=Kf(x) \quad \text{for some constant} \quad K\in\R \quad \text{and all} \quad x\in U_f.
\end{align}
By \eqref{Uf_disj_Ug} there is a point $x_0\in U_g\cap U_f$ so $AK=1$ and these coefficients are not zero. But then \eqref{fAg} implies  $U_g\subset U_f$ and \eqref{gKf} implies $U_f\subset U_g$; therefore, $U_g=U_f$ and $Z_g=Z_f$. The latter means that 
\[
f(x)=Ag(x), \, g(x)=Kf(x) \quad \text{if} \quad x\in Z_g=Z_f
\]
by trivial reasons (all these values are zeros) so with \eqref{fAg} and \eqref{gKf} these identities are valid on the entire $\R=U_f\cup Z_f=U_g\cup Z_g$. 
In particular, it follows that $\{F, G, 1\}$ are linearly dependent on $\R$.
\end{proof}

\section{The Cauchy MVT with fixed asymmetric mean value}

In this section we consider the asymmetric case, i.e., in \eqref{eqn:cauchy.gen} 
we take 
\begin{align}\label{c}
\alpha, \, \beta \in (0,1) \quad \text{with } \quad \alpha\neq 1/2 \quad \text{and} \quad \beta=1-\alpha. 
\end{align}

The following proposition describes all pairs $(F,G)$ of two times continuously differentiable functions satisfying \eqref{eqn:cauchy.gen} under the assumption \eqref{c} on $\alpha,\beta$ in the intervals where $g=G'$ does not vanish.

\begin{proposition}
\label{cauchy1}
Let $(F,G)$ be a solution of the Problem~\ref{the-problem} with $\alpha,\beta$ satisfying \eqref{c} and $I=(p,q)$, $-\infty\leq p<q \leq +\infty$, be an interval where the derivative $g(x)$ does not vanish. If $F,G$ are two times continuously differentiable on $I$, then $\{F, G, 1\}$ are linearly dependent on $I$.
\end{proposition} 
\begin{proof}
With the change of variables $h=b-a$, $x=\alpha a+\beta b$, \eqref{eqn:cauchy.gen} yields
\begin{equation}\label{eqn_asym1}
[F(x+\alpha h)-F(x-\beta h)]\, g(x) = [G(x+\alpha h)-G(x-\beta h)]\, f(x)
\end{equation}
if $x\in I$ and $h>0$ such that $x+\alpha h$, $x-\beta h\in I$. The latter condition tells that \eqref{eqn_asym1} holds if $(h,x)$ lies within the open triangle 
\begin{align}\label{triangleT}
T := \bigl\{(h,x): \,0<h<q-p, p+\beta h<x<q-\alpha h \bigr\}.
\end{align}

\begin{figure}[h!]
  \caption{}
  \centering
    \includegraphics[width=0.5\textwidth]{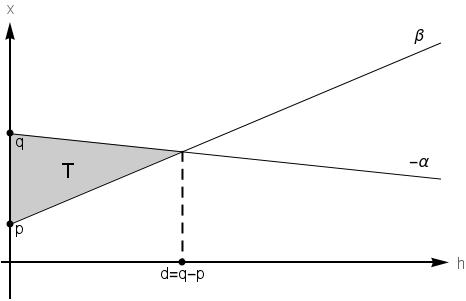}
\end{figure}

By differentiating both sides of \eqref{eqn_asym1} with respect to $h$ twice, we obtain the following relation in $T$
\begin{align}\label{eqn_asym2}
[\alpha^2 f'(x+\alpha h) - \beta^2 f'(x-\beta h)]\,g(x)= [\alpha^2 g(x+\alpha h)-\beta^2 g'(x-\beta h)]\,f(x).
\end{align}
All the functions are continiuous so \eqref{eqn_asym2} holds on the closure $\overline{T}$ as well, in particular, on the interval $\{h=0, p<x<q\}$. Therefore, with $\beta^2-\alpha^2=1-2\alpha\neq 0$ by \eqref{c}, we get $f'(x)g(x)=g'(x)f(x)$ for all $x\in I=(p,q)$. We can divide both sides by $g^2(x)$ and conclude that $(f/g)'=0$ on $I$. This implies that $f/g = A$ for some constant $A\in \R$, and $F'(x)=f(x)=Ag(x)=AG'(x)$, $x\in I$. After integration we get $F(x)=AG(x)+B(x)$, $x\in I$.
\end{proof}

The following theorem is the main result of this section.
%
\begin{theorem} 
\label{thcauchy1}
Let $(F,G)$ be a solution of the Problem~\ref{the-problem} with $\alpha,\beta$ satisfying \eqref{c}. If $F,G$ are two times continuously differentiable on $\R$, then $\{F, G, 1\}$ are linearly dependent on $\R$, i.e. there exist constants $A,B,C\in\R$ such that not all of them zeroes and 
\begin{equation}\label{lin.dep.1}
AF(x)+BG(x)+C=0 \quad \text{for all} \quad x\in\R.
\end{equation}
\end{theorem}
\begin{proof}
Consider the following cases:

\smallskip
\noindent
Case 1: $U_g = \emptyset$.

In this case $G$ is a constant on $\R$ and \eqref{eqn:cauchy.gen} holds for any differentiable function $F$. Hence \eqref{lin.dep.1} holds, for example, with $A=0$, $B=1$, $C=-G$ and thus $\{F, G, 1\}$ are linearly dependent on $\R$.

\smallskip
\noindent
Case 2: $U_g \neq \emptyset$ but $U_g \cap U_f=\emptyset$.

In this case Proposition \ref{prop:cauchy.gen.U_g<>empty} yields that $F$ is a constant on $\R$ and \eqref{eqn:cauchy.gen} holds for any differentiable function $G$. Hence \eqref{lin.dep.1} holds, for example, with $A=1$, $B=0$, $C=-F$ and thus $\{F, G, 1\}$ are again linearly dependent on $\R$.

\smallskip
\noindent
Case 3: $U_g \cap U_f \neq \emptyset$.

In this case Propositions \ref{cauchy1} and \ref{prop:X} immediately imply that $\{F, G, 1\}$ are linearly dependent on ~$\R$.
\end{proof}
%
%
\section{The Cauchy MVT with symmetric mean value} 

In this section we consider the problem of describing all pairs $(F,G)$ of smooth functions for which the mean value in \eqref{eqn:cauchy.gen} is taken at the midpoint of the interval. Our first result gives a necessary (and also sufficient in case $\{1,F,G\}$ are not linearly dependent) condition on such pairs in the intervals where $g=G'$ does not vanish.

\begin{proposition} 
\label{cauchy2}
Assume that $F,G: \R \to \R$ are three times differentiable functions with derivatives $F'=f$, $G'=g$. Let $I\subset \R$ be such an interval that $g\neq 0$ for all $x\in I$ and \eqref{eqn:cauchy.sym1} holds
%
for all $a,b \in I$. Then there exist constants $A, K\in\R$ and $x_{0}\in I$ such that 
\begin{equation} \label{f,g}
f(x) = \bigg(A + K \int_{x_{0}}^{x}\frac{dt}{g^{2}(t)}\bigg) \, g(x) \quad \text{for all} \quad x \in I.
\end{equation}
%
Moreover, if \eqref{f,g} holds with $K\neq0$, then \eqref{eqn:cauchy.sym1} holds if and only if  
\begin{equation} \label{integralcondition}
\int_{x-h}^{x+h} g(t)\, \Bigg(\int_{x_{0}}^{t}\frac{du}{g^{2}(u)} \Bigg) \,dt = \Bigg(\int_{x-h}^{x+h} g(t) dt\Bigg) \Bigg(\int_{x_{0}}^{x}\frac{du}{g^{2}(u)} \Bigg)  \end{equation} 
for all $  x,h \in \R$ such that $x, x+h, x-h\in I$.
\end{proposition}
\begin{proof} 
With the change of variables $x=\frac{a+b}{2}$, $h=\frac{b-a}{2}$, we can rewrite
\eqref{eqn:cauchy.sym1} as 
\begin{equation} 
\label{modifcauchy}
[F(x+h)-F(x-h)]g(x)= [G(x+h)-G(x-h)]f(x) 
\end{equation}
for all $x,h\in \R$ with the property that $x,x+h,x-h \in I$. By differentiating this equality three times with respect to $h$, we get
\begin{equation*}
[f''(x+h)+f''(x-h)]\,g(x) = [g''(x+h)+g''(x-h)]\,f(x).
\end{equation*}
Setting $h=0$, we obtain
\[
0=f''(x)g(x)-f(x)g''(x) = \bigl(f'(x)g(x)-f(x)g'(x)\bigr)' \quad \text{for all} \quad x\in I,
\]
and thus $f'(x)g(x)-f(x)g'(x)=K$ for some constant $K$. Then $\bigl(\frac{f}{g}(x)\bigr)'= \frac{K}{g^2(x)}$, $x \in I$, 
and integration over $(x_0,x)$ with any $x_0\in I$ yields \eqref{f,g}.

Now assume \eqref{f,g} holds with a nonzero constant $K$. Then we have
\begin{align}
\nonumber
F(x+h) - F(x-h) &= \int_{x-h}^{x+h}f(t) dt \\ \nonumber
                &= \int_{x-h}^{x+h} \bigg(A + K\int_{x_0}^t \frac{du}{g^2(u)} \bigg)\,g(t) dt\\ \label{FCT_comp1}
                &= A\int_{x-h}^{x+h} g(t) dt + K\int_{x-h}^{x+h} g(t)\bigg(\int_{x_0}^t \frac{du}{g^2(u)} \bigg) dt
\end{align}
and
\begin{align}
\nonumber
[G(x+h)-G(x-h)]\frac{f(x)}{g(x)} =& \Bigg(\int_{x-h}^{x+h} g(t) dt\Bigg) \Bigg(A + K\int_{x_0}^x \frac{du}{g^2(u)}\Bigg)\\\label{FCT_comp2}
                                 =& A\int_{x-h}^{x+h} g(t) dt + K\Bigg(\int_{x-h}^{x+h} g(t) dt\Bigg) \Bigg(\int_{x_{0}}^{x}\frac{du}{g^{2}(u)} \Bigg).
\end{align}
By comparing \eqref{FCT_comp1} and \eqref{FCT_comp2}, it is easy to see that \eqref{integralcondition} is equivalent to \eqref{eqn:cauchy.sym1}.
\end{proof} 
The following example illustrates that there are non-trivial functions satisfying \eqref{integralcondition} (and hence \eqref{eqn:cauchy.sym1}) on $\R$. 
\begin{example}
Consider $g(t)= e^t$ on $I=\R$ and let $A=0, K=1, x_0=0$. The integral condition \eqref{integralcondition} reads as the following identity
\begin{equation*}
\int_{x-h}^{x+h} e^t \, \Bigg(\int_{0}^te^{-2u} du\Bigg) \,dt= \Bigg(\int_{x-h}^{x+h}e^t dt\Bigg)\Bigg(\int_{0}^x e^{-2u}du\Bigg).
\end{equation*}
A direct computation gives $f(x) = \sinh(x)=\frac{e^x-e^{-x}}{2}$, and consequently,
\begin{align}\label{F=cosh}
F(x) = \cosh(x) = \frac{e^x+e^{-x}}{2}, \quad G(x) = e^x, \, x\in\R.
\end{align}
\end{example} 

We invite the interested reader to verify directly that the pair $(F,G)$ in \eqref{F=cosh} satisfies the relation \eqref{eqn:cauchy.sym1},  giving a non--trivial example of such pairs. 

Now we assume that $K\neq 0$ and analyze the property \eqref{integralcondition} for all $x,h\in\R$ such that $x,x+h,x-h\in I$. Differentiating it with respect to $h$, we obtain 
\[
g(x+h)\int_{x_0}^{x+h}\frac{du}{g^2(u)} + g(x-h)\int_{x_0}^{x-h}\frac{du}{g^2(u)} = [g(x+h)+g(x-h)]\int_{x_0}^x\frac{du}{g^2(u)}.
\]
Differentiation two more times with respect to $h$ gives
\[
g''(x+h)\int_{x_0}^{x+h}\frac{du}{g^2(u)} + g''(x-h)\int_{x_0}^{x-h}\frac{du}{g^2(u)} = [g''(x+h)+g''(x-h)]\int_{x_0}^x\frac{du}{g^2(u)},
\]
for all $x\in I$ and $h\in\R$ such that $x,x+h,x-h\in I$. Setting $h=x-x_0$ in these two equations, we obtain
\begin{align} \label{eqn_for_g}
g(2x-x_0)\int_{x_0}^{2x-x_0}\frac{du}{g^2(u)}  = [g(2x-x_0)+g(x_0)]\int_{x_0}^x\frac{du}{g^2(u)},
\end{align}
and
\begin{align} \label{eqn_for_g''}
g''(2x-x_0)\int_{x_0}^{2x-x_0}\frac{du}{g^2(u)} = [g''(2x-x_0)+g''(x_0)]\int_{x_0}^x\frac{du}{g^2(u)},
\end{align}
for all $x\in I$ with $2x-x_0 \in I$. Since $2x-x_0\in I$ and $g$ has no zeros in $I$, both sides of \eqref{eqn_for_g} do not vanish. By comparing \eqref{eqn_for_g''} and \eqref{eqn_for_g}, we get 
\begin{align}\label{frac_eq_g''}
\frac{g''(2x-x_0)}{g(2x-x_0)} = \frac{g''(2x-x_0)+g''(x_0)}{g(2x-x_0)+g(x_0)}
\end{align}
for all $x\in I$ such that $2x-x_0 \in I$. Denoting by $y(x) := g(2x-x_0)$ and $\lambda := \frac{4g''(x_0)}{g(x_0)}$, \eqref{frac_eq_g''} yields the second order differential equation $y''-\lambda y=0$, whose general real-valued solution (depending on the sign of $\lambda$), has the following form
\begin{align*}
g(x) &= Px+Q, \quad &&\text{if} \quad \lambda=0;\\
g(x) &= Pe^{\sqrt{\lambda} x}+Qe^{-\sqrt{\lambda} x} \quad &&\text{if} \quad \lambda=\mu^2, \mu>0;\\
g(x) &= P\sin(\sqrt{-\lambda} x)+Q\cos(\sqrt{-\lambda} x) \quad &&\text{if} \quad \lambda=-\mu^2, \mu>0,
\end{align*}
where $P$, $Q$ are real constants. Hence $G$ has one of the following forms
\begin{align}\label{G_quad}
G(x)&=Ax^2+Bx+C,\\\label{G_exp}
G(x)&=Ae^{\mu x}+Be^{-\mu x}+C, \quad \mu>0,\\\label{G_trig}
G(x)&=A\sin(\mu x)+B\cos(\mu x)+C, \quad \mu>0,
\end{align}
where $A,B,C$ are real constants. 
%
%

\begin{remark}
\label{rem:cauchy.sym}
Altogether, we come to the following conclusion:
on every interval $I\subset\R$ on which $G'\neq 0$, either $\{F, G, 1\}$ are linearly dependent, or $G$ and thus also $F$, cf. \eqref{f,g}, has one of the forms described in \eqref{G_quad}--\eqref{G_trig}.
\end{remark}
In the sequel, we call a function $G$ (resp. the pair $(F,G)$) to be of \textit{quadratic, exponential} or \textit{trigonometric type on $I$} if $G$ has (resp. both of $F$ and $G$ have) the form \eqref{G_quad}, \eqref{G_exp} or \eqref{G_trig}, respectively.
%

Consider the set $U_g$ and its representation, cf. \eqref{U_f,U_g}, \eqref{rep:U_g=sum.of.itvs}. The following lemma plays a crucial role in the analysis of the equation \eqref{eqn:cauchy.sym1}.

%
\begin{lemma}
\label{lem:cauchy.sym.f(q)=0}
Let $(p,q)\in\{I_{\sigma}\}_{\sigma\in\Sigma}$ be such that $p>-\infty$ and $f(p)=0$. 
Then $\{F, G, 1\}$ are linearly dependent on $[p,q)$.
\end{lemma}
\begin{proof}
$g(p)=0$ by Definition \eqref{rep:U_g=sum.of.itvs} so by Remark \ref{rem:cauchy.sym}, it is sufficient to consider the following cases.

\smallskip
\noindent
Case 1: $G$ is of quadratic type on $(p,q)$.  

Then $F$ is also of quadratic type on $(p,q)$, and since $f(p)=g(p)=0$, we have $F,G\in\text{span}\bigl\{1, (x-p)^2\bigr\}$. Thus $\{F, G, 1\}$ are linearly dependent on $(p,q)$. 

\smallskip
\noindent
Case 2: $G$ is of either exponential or trigonometric type on $(p,q)$.

First suppose that $G$ is of exponential type on $(p,q)$. Then so is $F$ and since the set of functions satisfying \eqref{eqn:cauchy.sym1} is invariant with respect to the addition of constant functions, we can assume, without loss of generality, that $F,G \in \text{span} \bigl\{e^{\,\mu (x-p)}, e^{-\mu (x-p)}\bigr\}$ for some $\mu \neq 0$. Hence there are real constants $u$, $v$ such that $F(x) = ue^{\mu(x-p)} + ve^{-\mu(x-p)}$, $x\in (p,q)$. Since $F'(p)=f(p)=0$, we get $u=v$ and thus $F(x)=2u\cosh(\mu(x-p))$. The same argument for $G$ explains that $G(x)=2w\cosh(\mu(x-p))$ for some real $w$, and consequently $F$ and $G$ are multiples of the same function $\cosh(\mu(x-p))$.

If $G$ is of trigonometric type, then by the same way as above, we can conclude that $F$ and $G$ are multiples of the same function $\cos(\mu(x-p))$, implying that $\{F, G, 1\}$ are linearly dependent on $[p,q)$.
%
%
%
%
%
\end{proof}
%
%

\paragraph{Proof of Theorem~\ref{thm:sym.cauchy.main}} Consider the set $U_g$ defined in \eqref{U_f,U_g}.
If 
$U_g=\emptyset$, then $g\equiv 0$ on $\R$, and thus $G$ is identically constant on $\R$. In this case $F$ can be an arbitrary differentiable function on $\R$ and thus $\{1, F, G\}$ are linearly dependent on $\R$. If $U_g=\R$, then it follows (cf. Remark \ref{rem:cauchy.sym}) that either $\{1, F, G\}$ are linearly dependent
or $G$ has one of the forms \eqref{G_quad}--\eqref{G_trig} on the whole of $\R$. Moreover, we get the same conclusion if $U_g\cap U_f=\emptyset$ (cf. Proposition \ref{prop:cauchy.gen.U_g<>empty}). 

Next, let us assume that $U_g\cap U_f \neq \emptyset$ and $U_g$ is a proper subset of $\R$. Consider the representation ~\eqref{rep:U_g=sum.of.itvs}. It is clear (cf. Remark \ref{rem:cauchy.sym}) that the index set $\Sigma$ can be split into disjoint subsets as $\Sigma = \Sigma_{\textrm{lr}} \cup \Sigma_{\textrm{q}} \cup \Sigma_{\textrm{t}} \cup \Sigma_{\textrm{e}}$, where
\begin{align*}
\Sigma_{\textrm{lr}} :=& \, \bigl\{ \sigma\in\Sigma: \, \{F, G, 1\} \quad \text{are in linear relationship on} \; I_{\sigma} \bigr\},\\
\Sigma_{\textrm{q}} :=& \, \bigl\{\sigma\in\Sigma: \, (F, G) \quad \text{are of quadratic type on } \,I_{\sigma} \bigr\},\\
\Sigma_{\textrm{t}} :=& \, \bigl\{\sigma\in\Sigma: \, (F, G) \quad \text{are of trigonometric type on } \, I_{\sigma} \bigr\},\\
\Sigma_{\textrm{e}} :=& \, \bigl\{\sigma\in\Sigma: \, (F, G) \quad \text{are of exponential type on } \, I_{\sigma}\bigr\}.
\end{align*}

\smallskip
\noindent
\textit{Claim 1}. If $\Sigma_{\textrm{lr}} \neq \emptyset$, then $\Sigma_{\textrm{lr}} = \Sigma$.

\smallskip
\noindent
Proof. Assume $\Sigma_{\textrm{lr}}$ is a proper subest of $\Sigma$. Then there exists $\sigma_2\in\Sigma$ such that $\sigma_2\notin\Sigma_{\textrm{lr}}$. Since $\Sigma_{\textrm{lr}} \neq \emptyset$, there is $\sigma_1 \in \Sigma_{\textrm{lr}}$ and $A_1\in\R$ such that $f(x)=A_1g(x)$ on $x\in I_{\sigma_1}$. Consider all $x,h\in\R$ such that $x+h \in I_{\sigma_2}$ and $x\in I_{\sigma_1}$. Using \eqref{eqn:cauchy.sym1} for $a=x-h$ and $b=x+h$, and recalling that $g\neq 0$ on $I_{\sigma_1}$, we get
\begin{align}
F(x+h)-A_1G(x+h) = F(x-h) - A_1G(x-h).
\end{align}
Therefore, 
\begin{align*}
f(x+h) - A_1g(x+h) &= \frac{1}{2}\bigg(\frac{\partial}{\partial x} + \frac{\partial}{\partial h}\bigg) \, F(x+h) - \frac{A_1}{2}\bigg(\frac{\partial}{\partial x} + \frac{\partial}{\partial h}\bigg) \, G(x+h)\\
                   &= \frac{1}{2}\bigg(\frac{\partial}{\partial x} + \frac{\partial}{\partial h}\bigg) \bigl(F(x-h) - A_1G(x-h)\bigr) \\
                   &= 0,
\end{align*}
and thus $f(x+h) = A_1g(x+h)$ for all $x,h\in\R$ such that $x+h \in I_{\sigma_2}$ and $x\in I_{\sigma_1}$. From this it follows that $F$ and $G$ are in linear relationship on $I_{\sigma_2}$, that is, $\sigma_2 \in \Sigma_{\textrm{lr}}$, which leads to a contradiction. \qed

\smallskip
\noindent
\textit{Claim 2}. If $\Sigma_{\textrm{lr}} = \emptyset$, then only one of the index sets  $\Sigma_{\textrm{q}}$, $\Sigma_{\textrm{t}}$, $\Sigma_{\textrm{e}}$ is non-empty. 

\smallskip
\noindent
Proof. Let $\sigma\in \Sigma$ and $I_{\sigma} = (p,q)$. Since $U_g$ is proper subset of $\R$, one of $p,q$ is finite. We can assume $p>-\infty$. Then $g(p)=0$, and Lemma \ref{lem:cauchy.sym.f(q)=0} yields $f(p)\neq 0$. Hence using \eqref{eqn:cauchy.sym1} for $a=p-h$ and $b=p+h$ we get 
\begin{align}
\label{eqn:G(p+h)=G(p-h)}
G(p+h)=G(p-h) \quad \text{for all} \quad h\in\R,
\end{align}
so the graph of $G$ is symmetric with respect to the vertical line $y=p$. 

If $\sigma\in\Sigma_{\textrm{q}}$ or $\sigma\in\Sigma_{\textrm{e}}$, then $q=+\infty$ since the functions of quadratic type has exactly one and the functions of exponential type has at most one critical point. Therefore, if $\sigma\in\Sigma_{\textrm{q}}$, then 
$G \in \text{span} \{1, (x-p)^2\}$, $x\in\R$ and $\Sigma = \Sigma_{\textrm{q}}$. Similarly, it follows from \eqref{eqn:G(p+h)=G(p-h)} that if $\sigma\in\Sigma_{\textrm{e}}$, then $\Sigma = \Sigma_{\textrm{e}}$. 

Next, assume $\Sigma_{\textrm{lr}} = \Sigma_{\textrm{q}} = \Sigma_{\textrm{e}} = \emptyset$. Then $\Sigma = \Sigma_{\textrm{t}}$ and let $\sigma\in\Sigma_{\textrm{t}}$. Since $G$ is of trigonometric type on $I_{\sigma}=(p,q)$, we must have $q<+\infty$. So $g(p)=g(q)=0$ and it follows as in the proof of Lemma \ref{lem:cauchy.sym.f(q)=0} that there are real constants $u$, $v$ such that  
\begin{equation}\label{G_in_(p,q)}
G(x) = u+v\cos\bigl(\pi\frac{x-p}{q-p}\bigr), \quad x\in (p,q).
\end{equation}
%
%
%
%
Using \eqref{eqn:G(p+h)=G(p-h)} we obtain that \eqref{G_in_(p,q)} holds on the whole of $\R$. \qed
%
%
%

Since $U_g\neq\emptyset$, at least one of $\Sigma_{\textrm{lr}}$, $\Sigma_{\textrm{q}}$, $\Sigma_{\textrm{t}}$, $\Sigma_{\textrm{e}}$ is non-empty. If $\Sigma_{\textrm{lr}} \neq \emptyset$, then Claim 1 and Proposition \ref{prop:X} imply that $\{F, G, 1\}$ are linearly dependent on $\R$. If $\Sigma_{\textrm{lr}} = \emptyset$, then Claim 2 yields that one of the possibilities $(b)$ -- $(d)$ holds. \qed

\section{Final Remarks}

As a consequence of our main result we can give a partial answer to following  still open question of Sahoo and Riedel (cf. \cite[Section 2.7]{Sahoo-Riedel} for an equivalent formulation). 

\smallskip

\paragraph{Problem} \textit{Find all functions $F,G,\phi,\psi:\R \to \R$ satisfying 
\begin{equation}
\label{op:Sahoo-Riedel}
[F(x)-F(y)]\,\phi\bigl(\frac{x+y}{2}\bigr) = [G(x)-G(y)]\,\psi\bigl(\frac{x+y}{2}\bigr) 
\end{equation}
for all $x,y\in\R$.}

We provide a partial result to this problem under certain assumptions on the unknown functions. First let us make the change of variables $s=\frac{x+y}{2}$, $t=\frac{x-y}{2}$ and write \eqref{op:Sahoo-Riedel} equivalently as
\begin{equation}
\label{op:Sahoo-Riedel1}
[F(s+t)-F(s-t)]\,\phi(s) = [G(s+t)-G(s-t)]\,\psi(s), \quad s,t\in\R.
\end{equation}


\begin{theorem}
Let $F,G:\R\to\R$ be three times differentiable and $\phi,\psi:\R\to\R$ be an arbitrary functions satisfying \eqref{op:Sahoo-Riedel1} on $\R$. If either $\phi\neq0$ or $\psi\neq0$ on $\R$, then one of the following possibilities holds:
\begin{enumerate}[(a)]
\item there exist constants $A_0,A_1,A_2\in\R$ such that for all $s\in\R$, we have $A_0+A_1F(s)+A_2G(s)=0$ and $G'(s)\,[A_1\psi(s)+A_2\phi(s)]=0$;
\item there exist constants $A_0,A_1,A_2,B_0,B_1,B_2\in\R$ such that for all $s\in\R$, we have $F(s)=A_0+A_1s^2+A_2s^2$, $G(s)=B_0+B_1s+B_2s^2$ and 
\[
(A_1+2A_2s)\phi(s)=(B_1+2B_2s)\psi(s);
\]
\item there exists $\mu\neq0$ and constants $A_0,A_1,A_2,B_0,B_1,B_2\in\R$ such that for all $s\in\R$, we have $F(s)=A_0+A_1e^{\mu s}+A_2e^{-\mu s}$, $G(s)=B_0+B_1e^{\mu s}+B_2e^{-\mu s}$ and 
\[
(A_1e^{\mu s}-A_2e^{-\mu s})\phi(s) = (B_1e^{\mu s}-B_2e^{-\mu s})\psi(s); 
\]
\item there exists $\mu\neq0$ and constants $A_0,A_1,A_2,B_0,B_1,B_2\in\R$ such that for all $s\in\R$, we have $F(s)=A_0+A_1\sin(\mu s)+A_2\cos(\mu s)$, $G(s)=B_0+B_1\sin(\mu s)+B_2\cos(\mu s)$ and 
\[
[A_1\cos(\mu s)-A_2\sin(\mu s)]\,\phi(s) = [B_1\cos(\mu s)-B_2\sin(\mu s)]\,\psi(s). 
\]
\end{enumerate} 
\end{theorem}
\begin{proof}
Let $f,g$ be the derivatives of $F,G$, respectively and the sets $U_g$, $U_f$ (resp. $Z_g$, $Z_f$) be defined as in Section 3. Without loss of generality, assume that $\phi$ does not vanish on $\R$. By differentiating \eqref{op:Sahoo-Riedel1} with respect to $t$ and setting $t=0$ in the resulting equation, we get 
\begin{equation}
\label{op1}
f(s)\phi(s) = g(s)\psi(s), \quad s\in\R.
\end{equation}
For any $s\in U_g$ and $t\in\R$, by \eqref{op:Sahoo-Riedel1} and \eqref{op1}, we have
\begin{align*}
F(s+t)-F(s-t) &= [G(s+t)-G(s-t)]\,\frac{\psi(s)}{\phi(s)}\\
              &= [G(s+t)-G(s-t)]\,\frac{f(s)}{g(s)},
\end{align*}
and thus
\begin{equation}
\label{op2}
[F(s+t)-F(s-t)]\, g(s) = [G(s+t)-G(s-t)]\,f(s), \quad s\in U_g, t\in\R. 
\end{equation}
On the other hand, observe that we have $Z_g\subset Z_f$ by \eqref{op1} since $\phi\neq0$ on $\R$. So \eqref{op2} holds for all $s\in U_g\cup Z_g=\R$. Therefore, Theorem~\eqref{thm:sym.cauchy.main} can be applied to \eqref{op2} and the four characterizations follows immediately. 
\end{proof}

\medskip 

It is likely that the methods of this paper work for related equations when we replace the linear mean $\alpha a + (1-\alpha) b$ by the $p$-mean $M^p_\alpha(a,b) = (\alpha a^p + (1-\alpha) b)^{\frac{1}{p}}$, for $a,b \geq 0$. Here $M^p_\alpha(a,b)$ is defined for all 
values of $ p \neq 0$. For $p=0$ the corresponding mean is defined by $M^0_\alpha(a,b) = a^\alpha b^{1-\alpha}$. Moreover, for $p \in \{ -\infty, \infty \}$ we can define $M^{-\infty}_\alpha(a,b)= \min \{ a, b \}$ and  $M^{\infty}_\alpha(a,b)= \max \{ a, b \}$. We intend to investigate this issue in a subsequent paper. 

\medskip

The essence of our approach is to reduce a functional equation to an ODE. For this strategy we need certain smoothness assumptions. It would be interesting to provide an alternative way that will not require this additional smoothness assumptions.  


\subsection*{Acknowledgment}
We would like to thank Professor J\"urg R\"atz for helpful conversations on the subject of this work.


\begin{thebibliography}{1}
%
\bibitem{Aczel}
J. Acz{\'e}l, \textit{A mean value property of the derivative of quadratic polynomials - without mean values and derivatives.} Math. Mag. \textbf{58} (1985), No. 1, 42--45. 
%
\bibitem{Aumann}
G. Aumann, \textit{\"Uber eine elementararen Zusammenhang zwischen Mittelwerten, Streckenrechnung und Kegelschnitten.} Tohoku Math. J. \textbf{42} (1936), 32--37.
%
\bibitem{Ebanks}
B. Ebanks, \textit{Generalized Cauchy difference equations II.} Proc. Amer. Math. Soc. \textbf{136} (2008), No. 11, 3911--3919.
%
\bibitem{Fechner-Gselmann}
W. Fechner, E. Gselmann, \textit{General and alien solutions of a functional equation and of a functional inequality.} Publ. Math. Debrecen \textbf{80} (2012), No. 1-2, 143--154.
%
\bibitem{Haruki}
S. Haruki, \textit{A property of quadratic polynomials.} Amer. Math. Monthly \textbf{86} (1979), No. 7, 577--579. 
%
\bibitem{Kannappan}
Pl. Kannappan, \textit{Rudin's problem on groups and a generalization of mean value theorem.} Aequationes Math. \textbf{65} (2003), No. 1-2, 82--92. 
%
\bibitem{Pales}
Z. P{\'a}les, \textit{On the equality of quasi-arithmetic and Lagrangian means.} J. Math. Anal. Appl. \textbf{382} (2011), 86--96.
%
\bibitem{Sablik}
M. Sablik, \textit{Taylor's theorem and functional equations.} Aequationes Math. \textbf{60} (2000), No. 3, 258--267. 
%
\bibitem{Sahoo-Riedel}
P.K. Sahoo, T. Riedel, \textit{Mean value theorems and functional equations.}  World Scieintific Publishing Co. Inc., River Edge, NJ, 1998.

\end{thebibliography}
\end{document}